\documentclass[11pt]{article}

\usepackage{amsmath,amssymb,amscd,amsthm}
\usepackage{ascmac}
\usepackage[all]{xy}
\usepackage{enumerate}
\usepackage{cancel}
\usepackage[dvips]{graphicx}
\usepackage{authblk}
\usepackage{comment}
\usepackage{mathrsfs}
\usepackage{array,arydshln}
\usepackage{geometry}
\usepackage{url}
\usepackage{ulem}
\theoremstyle{plain}
\newtheorem{thm}{Theorem}[section]
\newtheorem{pro}[thm]{Proposition}
\newtheorem{cor}[thm]{Corollary}
\newtheorem{lem}[thm]{Lemma}

\newtheorem{dfn-thm}[thm]{Definition-Theorem}
\newtheorem{dfn-pro}[thm]{Definition-Proposition}

\theoremstyle{definition}

\newtheorem{dfn}[thm]{Definition}

\newtheorem{asm}[thm]{Assumption}
\newtheorem{nota}[thm]{Notations}

\theoremstyle{remark}
\newtheorem{rmk}[thm]{Remark}

\newcommand{\bb}[1]{\mathbb{#1}}

\newcommand{\ind}{{\rm ind}}
\newcommand{\id}{{\rm id}}

\newcommand{\Pin}{{\rm Pin}}
\newcommand{\End}{{\rm End}}

\newcommand{\Tr}{{\rm Tr}}

\newcommand{\SW}{{\rm SW}}
\newcommand{\bra}[1]{\left(#1\right)}
\newcommand{\bbra}[1]{\left\{#1\right\}}
\newcommand{\bbbra}[1]{\left[#1\right]}

\newcommand{\tr}{{\rm tr}}

\newcommand{\R}{{\mathbb R}}
\newcommand{\HH}{{\mathbb H}}
\newcommand{\C}{{\mathbb C}}
\newcommand{\Z}{{\mathbb Z}}

\begin{document}
\title{Seiberg-Witten theory on finite covering spaces of spin $4$-manifolds}
\author{Tsuyoshi Kato, Doman Takata}

\date{\today}

\maketitle
\begin{abstract}

We compute the equivariant Bauer-Furuta degree, when a finite group acts
freely on a spin $4$-manifold.
In the case when the group is cyclic of order
power of two, Bryan gave a formula and its applications.
We have treated the case when the group has order
of odd degree. In particular we gave a formula
of the degree when the order is odd-prime.
Our approach is to use a representation-theoretic method
on finite dimensional approximations of the functional spaces.
\\\\
{\it Mathematics Subject Classification $(2020)$.} 57K41; 19L47, 20C05.
\end{abstract}

\section{Introduction}
The Bauer-Furuta theory provides with an algebro-topological
framework of a non-linear differential
mapping between Sobolev spaces
over a compact four-manifold.
In this article, we are interested in the variant with a finite group action. In particular we induce the formula of the 
equivariant degree of the monopole map 
for a class of finite group actions.
Even though it is a quite general formula, 
it involves various algebraic operations over representations
of the finite group. Hence, we start by describing a much simpler case and then, we naturally generalizes it 
to the equivariant-monopole case.

Let $X$ be a  Spin $4$-manifold equipped with a free 
Spin-action of a finite group $\Gamma$. We may assume $\sigma(M)\leq 0$.
We denote the orbit space $X/\Gamma$ by $M$. 
Throughout the paper, we assume $b^1(M)=0$, $b^+(M)>0$ and $X$ is connected.
Note that $b^+(X)$ is positive whenever $\#\Gamma >1$, because the Atiyah-Hitchin-Singer index is multiplicative by taking covering.
In this paper we present a representation-theoretic description of the degree $\alpha(SW_X)$ of the Seiberg-Witten (SW) map, that takes the value in 
the representation ring
$R(\Gamma \times \Pin(2)) \cong
R(\Gamma )\otimes R(\Pin(2))$.
Recall that $R(\Pin(2))$ is generated by two elements $c$ and $h$
that satisfy two relations $c^2=1$ and $ch=h$ (see Section \ref{rev}).

Our main theorem is the following.

\begin{thm}
Let $m=b^+(M)$ and $k=-\frac{\sigma(M)}{16}$.
For  $ \Gamma = \mathbb Z_p$ for odd prime $p$, we have the  formula:
\[
\alpha(\SW_X)=\bra{\frac{2^{(m-2k+1)p-2}-2^{m-2k-1}}{p}[L^2(\bb{Z}_p)]+2^{m-2k-1}\rho_{\rm triv}}(1-c).
\]
\end{thm}

Bryan solved a similar problem for several even order 
$\Gamma$'s (see the proof of Theorem $1.4$ in \cite{Bry}).
  In order to compare it  with our result, we impose extra conditions on Bryan's result.

\begin{thm}[\cite{Bry}]\label{bry}
Suppose that $(\bb{Z}_{2})^q$ freely acts on a Spin $4$-manifold $X$ preserving the Spin structure.
Let $M:= X/(\bb{Z}_2)^q$. If $b^+(X)\neq b^+(X/\left<g\right>)$ for any non-trivial $g\in (\bb{Z}_{2})^q$,
$$\alpha(\SW_X)=2^{2^q(m-2k+1)-2-q}[L^2((\bb{Z}_2)^q)](1-c).$$
\end{thm}

These two results cover the case of cyclic groups of prime order.
On the other hand, for the general case of finite groups $\Gamma$,
it is not so easy to induce a formula of the degree.
However, we can induce some constrains of the coefficients.

First, let us consider the non-prime cyclic case.
As a typical example, we consider $\Gamma =\Z_6$ case.
Note that we have the inductions of the representation ring
through the group homomorphisms
$i_2:\Z_2 \hookrightarrow \Z_6$ 
and $i_3:\Z_3 \hookrightarrow \Z_6$.
By using these inductions and the above theorems, we obtain the following.
Let $\rho_k$ be the irreducible representation of $\bb{Z}_6$ given by $\rho_l(\overline{n})=(e^{\frac{2\pi li}{6}})^{n}$.

\begin{cor}

Let $\alpha(\SW_X)=\sum_{l=0}^5 \beta_l\otimes \rho_l \in R({\rm Pin}(2))\otimes R(\bb{Z}_6)$.
Let $m_X= \ind (D_{{\rm AHS}}(X)) -1 = b^+(X)$, $2k_X= \ind (D_{{\rm Spin}}(X)) = - \frac{\sigma(X)}{8}$ and
\begin{align*}
& A:= 2^{m_X-2k_X+1-2-1}(1-c)
= 2^{m_X-2k_X-2}(1-c), \\
  & B:=\frac{2^{m_X-2k_X-1}+
 2^{\frac{m_X-2k_X+1}{3}-1}}{3}(1-c), \\
 &C:=\frac{2^{m_X-2k_X-1}-
 2^{\frac{m_X-2k_X+1}{3}-2}}{3}(1-c).
 \end{align*}
Then, the equations
\begin{align*}
\beta_2&= A-C-\beta_0+\beta_1,\\
\beta_3&= B-\beta_0, \\
\beta_4&= C-\beta_1, \\
\beta_5&= -A + 2C + \beta_0 - \beta_1
\end{align*}
hold in $R(\Pin(2))$.

In particular
$\alpha(\SW_X)=\sum_{l=0}^5 \beta_l\otimes \rho_l$
is determined by the two elements $\beta_0$ and $\beta_1$.
\end{cor}

Second, let us consider the general case of finite groups 
$\Gamma$ with odd order.
Let us denote
$\alpha(SW_X) = \alpha_0(X)-\widetilde{\alpha}_0(X)c+\sum_{k=1}^\infty \alpha_k(X)h^k \in R(\Gamma \times \Pin(2)) $.
We verify the following general formula.

\begin{thm}\label{thm main thm}
Suppose that the order of $\Gamma$ is odd. Then,
$$\alpha_0(X)+\widetilde{\alpha}_0(X)=
2^{m-2k}\bra{\frac{2^{(\#\Gamma-1)(m-2k+1)}-1}{\#\Gamma}[L^2(\Gamma)]+\rho_{\rm triv}}.$$
\end{thm}

These are the first step to compute the degree of the equivariant SW map. Our project is to seek for a nice class of finite groups and
Spin $4$-manifolds
that allows us to compute the degree concretely.
This is the next subject of our study and will be developed later.

\section{Review of 
finite dimensional approximation of the SW map }\label{rev}
Let us recall the construction of the monopole map \cite{BF}.
Let $M$ be a Spin $4$-manifold, and 
 $S^\pm$ be the spinor bundles.
The SW map is given by
\begin{gather*}
\SW \colon i\Omega^1(M)\oplus \Gamma(S^+)\to i(\Omega^0\oplus\Omega^+)(M)\oplus \Gamma(S^-)\\
\SW(a,\phi)= (d^*a, d^+a + q(\phi), (D+a)(\phi)),
\end{gather*}
where
$D$ is the Dirac operator  and  
$q(\phi)$ is the trace-free part of $(\phi\otimes\phi^*)\in\Gamma(\End(S^+))$ that is identified with a self-dual two form via the Clifford multiplication.
Let $\bf V$ be the $L^2_k$-completion ($k\geq 4$) of the source of 
the SW map and $\bf W$ be the $L^2_{k-1}$-completion of the target.
The index of the linearized operator of the SW map is given by 
$$
\ind D_{\text{Spin}} - \ind D_{\rm AHS} 
=  - \frac{\sigma(M)}{16} - (1-b^1(M) + b^+(M)).
$$

Taking a finite dimensional approximation of $\SW: \bf V \to \bf W$,
 we have a properly $G$-equivariant  map
\cite{BF}
\begin{equation}\label{eq:f}
f\colon V\to W,
\end{equation}
where $V$, $W$ are some finite dimensional subspaces of 
the Sobolev spaces.
Let $\tilde{\R}$ be the one-dimensional real representation of $G$ on which $S^1\subset G$ acts trivially and $j$ acts by multiplication of $-1$.
Let $\HH$ be the quaternion representation via left multiplication.

We have identifications as $G$-modules, $V\cong \tilde{\R}^{m}\oplus\HH^{n+k}$ and $W\cong \tilde{\R}^{m+b} \oplus\HH^n$,
where $k=\ind_{\HH}D_{{\rm Spin}}=-\sigma(M)/16$  and $b=b_+(M)$.
We consider the complexification of $f$.
$$
f_c\colon \tilde{\C}^m\oplus \HH_c^{n+k}\to \tilde{\C}^{m+b}\oplus \HH_c^{n},
$$
where $\tilde{\C}=\tilde{\R}\otimes_{\R} \C$ and $\HH_c=\HH\otimes_{\R}\C = \HH\oplus\HH$. 
The map $f_c$ is also $G$-equivariant and proper.

Recall  some facts on the complex representation ring $R(G)$:
\begin{itemize}
\item $R(G)$ is generated by $c=[\tilde{\C}]$ and $h=[\HH]$ subject to the relations $c^2=1$ and $ch=h$.
More expplicitly 
\[
R(G) = \Z[h,c]/(c^2-1,ch-h)=\Z[h]\oplus \Z c.
\]
\item For $\beta\in R(G)$, let $\lambda_{-1}\beta$ denote the alternating sum $\sum(-1)^i\lambda^i\beta$ of exterior powers.
Then $\lambda_{-1}c=1-c$ and $\lambda_{-1} h = 2-h$.
\end{itemize}

Apply the $K_G$-functor
\[
f_c^*: K_G(\tilde{\C}^{m+b}\oplus \HH_c^{n})
\to K_G(\tilde{\C}^m\oplus \HH_c^{n+k}).
\]
Recall that 
for a complex representation space $H$ of $G$, 
$K_G(H)$ is a free $R(G)$-module  generated by the Bott element.
Hence, $f_c^*$ is represented by the multiplication of an element
$\alpha_f \in R(G)$ in the representation ring.

\begin{dfn}
The $K_G$-theoretic degree of $f: V \to W$ is given by
$\alpha_f \in R(G)$.
\end{dfn}

\begin{pro}\cite{Fur}
The $K_{\Pin(2)}$-theoretic degree of the SW map is 
given by
\[
\alpha(SW_M) = 
2^{b-2k-1}(1-c).
\]
\end{pro}

\section{Finite-coverings}

Let $X$ be a Spin $4$-manifold equipped with a free Spin-action of a finite group $\Gamma$. We denote the orbit space $X/\Gamma$ by $M$. The Seiberg-Witten map on each manifold is denoted by $\SW_X$ and $\SW_M$.
We would like to connect the $K_{\Pin(2)}$-degree of $\SW_X$ and that of $\SW_M$ by using the following commutative diagram:
$$\begin{CD}
C^\infty(X,S_0\oplus \wedge^1) @>\SW_X>> C^\infty(X,S_1\oplus \wedge^0\oplus \wedge^+) \\
@A\pi^{*}_{0}AA @AA\pi^{*}_{1}A \\
C^\infty(M,S_0\oplus \wedge^1) @>\SW_M>> C^\infty(M,S_1\oplus \wedge^0\oplus \wedge^+) 
\end{CD}$$
Note that $\pi^{*}_{0}$ and $\pi^{*}_{1}$ are linear embeddings, and hence they are proper.

With a similar technique of the Peter-Weyl theorem for finite groups, we have the following isomorphism.

\begin{lem}\label{lem noncommutative Fourier}
For an arbitrary $\Gamma$-equivariant Hermite vector bundle $E$ of $X$, $C^\infty(X,E)$ can be decomposed as
$$\bigoplus_{\lambda\in\widehat{\Gamma}}V_\lambda\otimes C^\infty(M,X\times_\Gamma (E\otimes V_\lambda^*))$$
as a representation of $\Gamma$, where $V_\lambda$ is the representation space of the irreducible representation $\lambda$, and $V^*_\lambda$ is the dual representation space of $\lambda$. The same formula holds for not only $C^\infty$, but also other types of function spaces including $C^0$, $L^2$ and so on.
\end{lem}
\begin{proof}
We construct an isomorphism and sketch the proof.
First, we define a $\Gamma$-equivariant vector bundle $X\times_\Gamma (E\otimes l^2(\Gamma))$ over $M$. It is the quotient space of $E\otimes l^2(\Gamma)$
by the $\Gamma$-action
$$(\Gamma\text{-action on }E )\otimes(\text{right regular representation}).$$
The fiber at $m\in M$ can be identified with $l^2(\Gamma, E_{\widetilde{m}})$, where $\widetilde{m}\in X$ is a chosen lift of $m$. An element of the fiber $\phi \in l^2(\Gamma, E_{\widetilde{m}})$ is identified with 
$\psi \in l^2(\Gamma, E_{\gamma_0\widetilde{m}})$ if the equation
$$\gamma_0\cdot \phi(\gamma\gamma_0)=\psi(\gamma)$$
holds. The $\Gamma$-action on $X\times_\Gamma (E\otimes l^2(\Gamma))$ is defined by
$$\id_E \otimes (\text{left regular representation}).$$
Notice that the $\Gamma$-action on $M$ is trivial.

Next, we construct an isomorphism $\Phi$ between $C^\infty(X,E)$ and $C^\infty(M,X\times_\Gamma (E\otimes l^2(\Gamma)))$.
For $s\in C^\infty(X,E)$, we define $\Phi(s)\in C^\infty(M,X\times_\Gamma (E\otimes l^2(\Gamma)))$ by
$$\Phi(s)(m):=\Bigl[\Gamma\ni \gamma\mapsto \gamma^{-1}s(\gamma\widetilde{m})\in E|_{\widetilde{m}}
\Bigr].$$
One can easily prove that $\Phi$ is well-defined and it is $\Gamma$-equivariantly isomorphic.

Let us construct an isomorphism between $C^\infty(X,E)$ and $\bigoplus_{\lambda\in\widehat{\Gamma}}V_\lambda\otimes C^\infty(M,X\times_\Gamma (E\otimes V_\lambda^*))$. By the Peter-Weyl theorem, $l^2(\Gamma)\cong \bigoplus_{\lambda\in\widehat{\Gamma}}V_\lambda\otimes V_\lambda^*$. Since the right regular representation acts on the $V_\lambda$-factor trivially, we have an isomorphism
$$C^\infty(M,X\times_\Gamma (E\otimes l^2(\Gamma)))\cong\bigoplus_{\lambda\in\widehat{\Gamma}}V_\lambda\otimes C^\infty(M,X\times_\Gamma (E\otimes V_\lambda^*)).$$
Since the left regular representation corresponds to $\rho_\lambda\otimes \id$ on the summand $V_\lambda\otimes V_\lambda^*$, the above isomorphism is equivariant.

\end{proof}

We work on the representation ring $R(\Gamma\times {\rm Pin}(2))\cong R(\Gamma)\otimes R({\rm Pin}(2))$. In the representation ring, we write the representation $V_\lambda$ as $\lambda$. 
We use additive notation for direct sum, multiplicative notation for interior tensor product of representations, and tensor product notation for exterior tensor product of representations (a tensor product of representations of different groups).
Recall that $R({\rm Pin}(2))$ is generated by $h$ and $c$. We denote the trivial representation by $\rho_{\rm triv}$.

\begin{thm}\label{pro SW of X VS SW of M}
$\alpha(\SW_M)$ and $\alpha(\SW_X)$ satisfy the following equation in $R(\Gamma\times {\rm Pin}(2))$.
\begin{align*}
&\prod_{\lambda\neq \rho_{\rm triv}}\bigwedge^*\bra{\lambda\otimes\bra{2N_\lambda\cdot \dim( V_\lambda^*)h+ M_\lambda\cdot \dim( V_\lambda^*)c}} \\
& \qquad  \qquad \qquad 
\times \bra{\alpha(\SW_X)\cdot \prod_{\lambda\neq \rho_{\rm triv}}\bigwedge^*\bra{V_\lambda\otimes\bra{\ind(D_{\rm Spin}(M))\cdot \dim( V_\lambda^*)h}}
}\\
& = \ \prod_{\lambda\neq \rho_{\rm triv}}\bigwedge^*\bra{\lambda\otimes\bra{2N_\lambda\cdot \dim( V_\lambda^*)h+ M_\lambda\cdot \dim( V_\lambda^*)c}} \\
& \qquad  \qquad \qquad 
\times \bra{\prod_{\lambda\neq \rho_{\rm triv}}\bigwedge^*\bra{\lambda\otimes\bra{ \ind(D_{\rm AHS}(M))\cdot \dim( V_\lambda^*)c}}\cdot \alpha(\SW_M)},
\end{align*}
where $N_\lambda$ and $M_\lambda$ are natural numbers appearing in the definition of finite-dimensional approximations.
\end{thm}

\begin{proof}
Let us first outline the idea of the proof. We will justify this formal argument by using finite-dimensional approximations. 

We denote $\Gamma\times {\rm Pin}(2)$ by $G$.
By the above commutative diagram and the Thom isomorphisim, we obtain the following diagram on topological $K$-theory:
\begin{equation}\label{eqn comm diagram}
\xymatrix{
R(G) \ar_{e_0}[ddd] \ar^{\cong}[rd]&&&
R(G) \ar_{\alpha(\SW_X)}[lll] \ar^{e_1}[ddd] \ar_{\cong}[ld]\\
&K^0_{G}(C^\infty(X,S_0\oplus \wedge^1))\ar_{(\pi_0^*)^*}[d] & 
K^0_{G}(C^\infty(X,S_1\oplus \wedge^0\oplus \wedge^+)) \ar_{\SW_X^*}[l] \ar^{(\pi_1^*)^*}[d] \\
&K^0_{G}(C^\infty(M,S_0\oplus \wedge^1)) &
K^0_{G}(C^\infty(M,S_1\oplus \wedge^0\oplus \wedge^+) ) \ar^{\SW_M^*}[l] \\
R(G) \ar^{\cong}[ru]&&&
R(G) \ar^{\alpha(\SW_M)}[lll] \ar_{\cong}[lu]},
\end{equation}
where $e_0$ and $e_1$ are the Euler classes corresponding to the orthogonal projections associated to the embeddings 
$\pi_0^*$ and $\pi_1^*$, respectively.
We can compute them thanks to Lemma \ref{lem noncommutative Fourier}. Formally, they can be written as
\begin{center}
$e_0=$``$\bigotimes_{\lambda\neq \rho_{\rm triv}}\bigwedge^* V_\lambda\otimes C^\infty(M,X\times_\Gamma(S_0\oplus \wedge^1)\otimes V_\lambda^*)$'', and

$e_1=$``$\bigotimes_{\lambda\neq \rho_{\rm triv}}\bigwedge^* V_\lambda\otimes C^\infty(M,X\times_\Gamma(S_1\oplus \wedge^0\oplus \wedge^+)\otimes V_\lambda^*)$''.
\end{center}
By the commutativity of the diagram, we obtain a formula
$\alpha(\SW_X)\cdot e_0=e_1\cdot \alpha(\SW_M)$. Dividing it by $e_0$, we obtain
$$\alpha(\SW_X)=e_0^{-1}\cdot e_1\cdot \alpha(\SW_M)$$
By the formula on the Euler classes $e(V\oplus W)=e(V)\cdot e(W)$, the coefficient $e_0^{-1}\cdot e_1$ should be
\begin{align*}
& \text{``}\bigotimes_{\lambda\neq \rho_{\rm triv}}\bigwedge^* V_\lambda\otimes \bbra{C^\infty(M,X\times_\Gamma(S_1\oplus \wedge^0\oplus \wedge^+)\otimes V_\lambda^*)
-C^\infty(M,X\times_\Gamma(S_0\oplus \wedge^1)\otimes V_\lambda^*)}\text{''} \\
& 
=\text{``}\bigotimes_{\lambda\neq \rho_{\rm triv}}\bigwedge^* V_\lambda\otimes
 (\ind(D_{\rm Spin}:C^\infty(S_0\otimes V_\lambda^*)
 \to C^\infty(S_1\otimes V_\lambda^*)) \\
& \qquad  \qquad \qquad 
-\ind(D_{\rm AHS}:C^\infty((\wedge^0\oplus \wedge^+)\otimes V_\lambda^*)\to 
C^\infty(\wedge^1\otimes V_\lambda^*)) )\text{''},
\end{align*}
where $\ind(\cdots)$ is an element of $R(G)$. Moreover, since $V_\lambda^*$ is flat, the index can be computed by the operators without coefficients. Strictly speaking, we cannot divide the equation by $e_0$. What we can do is to factorize the equation, and ``$e_0^{-1}\cdot e_1$'' will essentially appear as a coefficient.

Let us justify this infinite-dimensional argument. 
Let us consider an even-dimensional approximation of the above diagram following \cite{Fur}. We complexify all of section spaces in order to use the Bott periodicity theorem. This operation preserves the decomposition given in Lemma \ref{lem noncommutative Fourier}.

We use the following notations:
\begin{itemize}
\item $U(X):=$ a finite-dimensional approximation of $C^\infty(X,S_0\otimes_{\bb{R}} \bb{C})$,
\item $W(X):=$ a finite-dimensional approximation of $C^\infty(X,\wedge^1\otimes_{\bb{R}} \bb{C})$,
\item $U'(X):=$ a finite-dimensional approximation of $C^\infty(X,S_1\otimes_{\bb{R}} \bb{C})$, and 
\item $W'(X):=$ a finite-dimensional approximation of $C^\infty(X,(\wedge^0\oplus \wedge^+)\otimes_{\bb{R}} \bb{C})$,
\item $U_\lambda(X):=$ a finite-dimensional approximation of $C^\infty(M,X\times_\Gamma(S_0\otimes_{\bb{R}} \bb{C}\otimes_{\bb{C}} V_\lambda^*))$,
\item $W_\lambda(X):=$ a finite-dimensional approximation of $C^\infty(M,X\times_\Gamma(\wedge^1\otimes_{\bb{R}} \bb{C}\otimes_{\bb{C}} V_\lambda^*))$,
\item $U_\lambda'(X):=$ a finite-dimensional approximation of $C^\infty(M,X\times_\Gamma(S_1\otimes_{\bb{R}} \bb{C}\otimes_{\bb{C}} V_\lambda^*))$,
\item $W_\lambda'(X):=$ a finite-dimensional approximation of $C^\infty(M,X\times_\Gamma((\wedge^0\oplus \wedge^+)\otimes_{\bb{R}} \bb{C}\otimes_{\bb{C}} V_\lambda^*))$,

\item $U(M):=$ a finite-dimensional approximation of $C^\infty(M,S_0\otimes_{\bb{R}} \bb{C})$,
\item $W(M):=$ a finite-dimensional approximation of $C^\infty(M,\wedge^1\otimes_{\bb{R}} \bb{C})$,
\item $U'(M):=$ a finite-dimensional approximation of $C^\infty(M,S_1\otimes_{\bb{R}} \bb{C})$, and 
\item $W'(M):=$ a finite-dimensional approximation of $C^\infty(M,(\wedge^0\oplus \wedge^+)\otimes_{\bb{R}} \bb{C})$.
\end{itemize}
Since each section on $M$ corresponds to a $\Gamma$-invariant section on $X$, we have a natural isomorphism $U(M)=U_{\rho_{\rm triv}}(X)$ and so on.
By Lemma \ref{lem noncommutative Fourier}, $U(X)=\oplus_{\lambda\in \widehat{\Gamma}}V_\lambda\otimes U_\lambda(X)$ and so on.

The finite-dimensional approximation of the complexification of $\SW_X$ and $\SW_M$ is denoted by $F_X$ and $F_M$, respectively.
Thus the diagram (\ref{eqn comm diagram}) is approximated by the following commutative diagram:
$$
\xymatrix{
R(G) \ar_{e_0}[ddd] \ar^{\cong}[rd]&&&
R(G) \ar_{\alpha(F_X)}[lll] \ar^{e_1}[ddd] \ar_{\cong}[ld]\\
&K^0_{G}\bra{\bigoplus_{\lambda\in \widehat{\Gamma}}V_\lambda\otimes \bbra{U_\lambda(X)\oplus W_\lambda(X)}} \ar_{(\pi_0^*)^*}[d] & 
K^0_{G}\bra{\bigoplus_{\lambda\in \widehat{\Gamma}}V_\lambda\otimes \bbra{U'_\lambda(X)\oplus W'_\lambda(X)}} \ar_{F_X^*}[l] \ar^{(\pi_1^*)^*}[d] \\
&K^0_{G}(U_{\rho_{\rm triv}}(X)\oplus W_{\rho_{\rm triv}}(X)) &
K^0_{G}(U'_{\rho_{\rm triv}}(X)\oplus W'_{\rho_{\rm triv}}(X)) \ar^{F_M^*}[l] \\
R(G) \ar^{\cong}[ru]&&&
R(G). \ar^{\alpha(F_M)}[lll] \ar_{\cong}[lu]}
$$

Since the above approximations are given by the spectral decomposition of the Dirac operators, there are natural numbers $N_\lambda,M_\lambda$,
$$U_\lambda(X)\cong \bb{H}_c^{N_\lambda+\frac{1}{2}\ind(D_{\rm Spin}(M)\otimes V_\lambda^*)},\ \ \ 
W_\lambda(X)\cong \widetilde{\bb{C}}^{M_\lambda},$$
$$U'_\lambda(X)\cong \bb{H}_c^{N_\lambda},\ \ \ 
W'_\lambda(X)\cong \widetilde{\bb{C}}^{M_\lambda+\ind(D_{\rm AHS}(M)\otimes V_\lambda^*)},$$
where $\widetilde{\bb{C}}$ is the representation space of ${\rm Pin}(2)$ given by $e^{i\theta}j^{n}\mapsto (-1)^n$ for $e^{i\theta}\in S^1$ and $n\in \bb{Z}_2$, and $\bb{H}_c$ is the complexification of the natural representation of ${\rm Pin}(2)\subseteq SU(1)\subseteq \bb{H}$. 
Note the following facts.
\begin{itemize}
\item $\ind(D_{\rm Spin}(X))$ is even because $D_{\rm Spin}$ is $\bb{H}$-linear.
\item Since $V_\lambda^*$ is a flat bundle, thanks to the index theorem, $\ind(D_{\rm Spin}(M)\otimes V_\lambda^*)=D_{\rm Spin}(M)\cdot \dim(V_\lambda^*)$.
\item For the same reason, $\ind(D_{\rm AHS}(M)\otimes V_\lambda^*)=D_{\rm AHS}(M)\cdot \dim(V_\lambda^*)$.
\end{itemize}
Thus $e_0$ and $e_1$ are given by
$$e_0=\prod_{\lambda\neq \rho_{\rm triv}}\bigwedge^*\bra{\lambda\otimes\bra{\bbra{2N_\lambda+\ind(D_{\rm Spin}(M))\cdot \dim( V_\lambda^*)}h+M_\lambda c}},$$
$$e_1=\prod_{\lambda\neq \rho_{\rm triv}}\bigwedge^*\bra{\lambda\otimes\bra{2N_\lambda h+\bbra{M_\lambda+\ind(D_{\rm AHS}(M))\cdot \dim( V_\lambda^*)}c}}.$$


Thanks to these formulas, 

\begin{align*}
&\alpha(\SW_X)\cdot e_0-e_1\cdot \alpha(\SW_M)\\
&=\alpha(\SW_X)\cdot\prod_{\lambda\neq \rho_{\rm triv}}\bigwedge^*\bra{\lambda\otimes\bra{\bbra{2N_\lambda+\ind(D_{\rm Spin}(M))\cdot \dim( V_\lambda^*)}h+M_\lambda c}}\\
& \qquad  \qquad \qquad 
 -\prod_{\lambda\neq \rho_{\rm triv}}\bigwedge^*\bra{\lambda\otimes\bra{2N_\lambda h+\bbra{M_\lambda+\ind(D_{\rm AHS}(M))\cdot \dim( V_\lambda^*)}c}}\cdot \alpha(\SW_M) \\
&=\prod_{\lambda\neq \rho_{\rm triv}}\bigwedge^*\bra{\lambda\otimes\bra{2N_\lambda h+M_\lambda c}}
\times   \biggl(  \alpha(\SW_X)\cdot \bbra{\prod_{\lambda\neq \rho_{\rm triv}}\bigwedge^*\bra{\lambda\otimes\bra{\bbra{\ind(D_{\rm Spin}(M))\cdot \dim( V_\lambda^*)}h}}} \\
& \qquad  \qquad \qquad 
-\bbra{\prod_{\lambda\neq \rho_{\rm triv}}\bigwedge^*\bra{\lambda\otimes\bra{ \ind(D_{\rm AHS}(M))\cdot \dim( V_\lambda^*) c}}}\cdot \alpha(\SW_M) \biggr).
\end{align*}

Since $\alpha(\SW_X)\cdot e_0-e_1\cdot \alpha(\SW_M)=0$, the most right hand side is $0$.
\end{proof}

In order to deduce nontrivial results, we impose a condition on $\Gamma$.

\begin{asm}
The order of $\Gamma$ is odd.
\end{asm}

Thanks to this assumption, the common coefficient 
\[
\prod_{\lambda\neq \rho_{\rm triv}}\bigwedge^*\bra{\lambda\otimes\bra{2N_\lambda\cdot \dim( V_\lambda^*)h+ M_\lambda\cdot \dim( V_\lambda^*)c}}
\]
 does not vanish in the following sense. 

\begin{lem}
For any $\gamma\in \Gamma$, $\Tr_{(\gamma,J)}\bra{\prod_{\lambda\neq \rho_{\rm triv}}\bigwedge^*\bra{\lambda\otimes\bra{2N_\lambda\cdot \dim( V_\lambda^*)h+ M_\lambda\cdot \dim( V_\lambda^*)c}}}\neq 0$.
\end{lem}
\begin{proof}
The subgroup of $\Gamma$ generated by $\gamma$ is denoted by $\left<\gamma\right>$, and the natural inclusion $\left<\gamma\right>\hookrightarrow \Gamma$ is denoted by $i$.
The restriction to $\left<\gamma\right>$ of the $\Gamma$-representation $V_\lambda$ can be decomposed as $i^*V_\lambda=\oplus_{\mu\in \widehat{\left<\gamma\right>}} W_{\mu}^{\oplus n_{\lambda \mu}}$, where $W_\mu$ is the representation space of an irreducible representation $\mu$ of $\left<\gamma\right>$. In the representation ring, it can be written as $i^*\lambda=\sum_{\mu\in \widehat{\left<\gamma\right>}} n_{\lambda \mu}\mu$.

Then
\begin{align*}
&i^*\bra{\prod_{\lambda\neq \rho_{\rm triv}}\bigwedge^*\bra{\lambda\otimes\bra{2N_\lambda\cdot \dim( V_\lambda^*)h+ M_\lambda\cdot \dim( V_\lambda^*)c}}}\\
&=\prod_{\lambda\neq \rho_{\rm triv}}\bigwedge^*\sum_{\mu\in \widehat{\left<\gamma\right>}} n_{\lambda \mu}\mu\otimes\bra{\bra{2N_\lambda\cdot \dim( V_\lambda^*)h+ M_\lambda\cdot \dim( V_\lambda^*)c}}\\
&=\prod_{\lambda\neq \rho_{\rm triv},\mu\in \widehat{\left<\gamma\right>}}\bra{\bigwedge^*\mu\otimes h}^{2n_{\lambda \mu}N_\lambda\cdot \dim( V_\lambda^*)}\otimes 
\prod_{\lambda\neq \rho_{\rm triv},\mu\in \widehat{\left<\gamma\right>}}\bra{\bigwedge^*\mu\otimes c}^{n_{\lambda \mu}M_\lambda\cdot \dim( V_\lambda^*)}.
\end{align*}
Thus, it suffices to prove that $\Tr_{(\gamma,J)}\bra{\bigwedge^*{\mu}\otimes h}\neq 0$ and $\Tr_{(\gamma,J)}\bra{{\mu}\otimes c}\neq 0$.

Since $\bigwedge^*{\mu}\otimes h=1\otimes 1-{\mu}\otimes h+\mu^2\otimes 1=(1+\mu^{2})\otimes 1-{\mu}\otimes h$,
\begin{align*}
\tr_{(\gamma,J)}(\bigwedge^*{\mu}\otimes h)
&= \tr_\gamma(1+\mu^{2})-\tr_\gamma(\mu)\tr_J(h) \\
&=1+\mu(\gamma)^2.
\end{align*}
We denote the order of $\gamma$ by ${\rm ord}(\gamma)$.
Since $\#\Gamma$ is odd, ${\rm ord}(\gamma)$ is also odd. Since $\mu(\gamma)$ is an ${\rm ord}(\gamma)$-th root, its square is not $-1$. Thus $\tr_{(\gamma,J)}(\bigwedge^*{\mu}\otimes h)\neq 0$.

In much the same way, the latter statement can be proved.
\end{proof}

By the same calculation, we will obtain a nontrivial result on $\alpha(\SW_X)$ from Theorem \ref{pro SW of X VS SW of M}.

\begin{nota}

Since $R(\Gamma\times {\rm Pin}(2))\cong R(\Gamma)\otimes R( {\rm Pin}(2))$, $\alpha(\SW_X)$ can be written as
$$\alpha_0(X)-\widetilde{\alpha}_0(X)c+\sum_{k=1}^\infty \alpha_k(X)h^k,$$
where $\alpha_0(X)$, $\widetilde{\alpha}_0(X)$, $\alpha_k(X)\in R(\Gamma)$ and all but a finite number of $\alpha_k(X)$'s are $0$.
\end{nota}
\begin{thm}[\cite{Fur}]\label{Fur}
The restriction of $\alpha(\SW_X)$ to ${\rm Pin}(2)$ is given by $2^{m(X)-2k(X)-1}(1-c)$, where 
$2k(X)=-\frac{\sigma(X)}{8}=\ind(D_{\rm Spin}(X))$ and $m(X)=b^+(X)=\ind(D_{\rm AHS}(X))-1$.
In particular, $\tr_{e}(\alpha_0(X))-\tr_{e}(\widetilde{\alpha}_0(X))=0$.
\end{thm}

The result of this section is to give some new information on $\alpha_0(X)$ and $\widetilde{\alpha}_0(X)$.

\begin{lem}\label{lem3.7}
For an odd number $n$, $\prod_{k=1}^{n-1}\bbbra{1+\exp\bra{\frac{2k\pi i}{n}}}=1$.

\end{lem}
\begin{proof}
Let $\zeta$ be a primitive $n$-th root. Let $\sigma_k$ be the $k$-th elementary symmetric polynomial of $(n-1)$ variables. Then
$$\prod_{k=1}^{n-1}(1+\zeta^k)=\sum_{k=0}^{n-1}\sigma_k(\zeta,\zeta^2,\cdots,\zeta^{n-1}).$$
Let $p_k$ be the $k$-th power sum $p_k(X_1,X_2,\cdots):=X_1^k+X_2^k+\cdots$. We prove that $\sigma_k(\zeta,\zeta^2,\cdots,\zeta^{n-1})=(-1)^k$ by induction on $k$. Note that $p_k(\zeta,\zeta^2,\cdots,\zeta^{n-1})+1=0$ if $1\leq k\leq n-1$, in particular $\sigma_1(\zeta,\zeta^2,\cdots,\zeta^{n-1})=-1$.
By Newton's identity,
\begin{align*}
k\sigma_k(\zeta,\zeta^2,\cdots,\zeta^{n-1})
&=\sum_{j=1}^k(-1)^{j-1}\sigma_{k-j}(\zeta,\zeta^2,\cdots,\zeta^{n-1})p_j(\zeta,\zeta^2,\cdots,\zeta^{n-1})\\
&=\sum_{j=1}^k(-1)^{j-1}(-1)^{k-j}(-1)\\
&=(-1)^kk.
\end{align*}

Thus, 
$$\prod_{k=1}^{p-1}(1+\zeta^k)=\sum_{i=0}^{n-1}(-1)^k=1.$$
\end{proof}

For each $n\in \bb{N}$, $L^2(\bb{Z}_n)$ is a unitary representation of $\bb{Z}_n$. Thus it defines an element of $R(\bb{Z}_n)$. The corresponding element to it is denote by $[L^2(\bb{Z}_n)]$.
 
\begin{cor}\label{cor lem on cyclic gp}
If $n$ is an odd number, for any generator $\gamma\in \bb{Z}_n$,
$$\Tr_{(\gamma^k,J)}\bra{\bigwedge^* [L^2(\bb{Z}_n)]\otimes h}
=\begin{cases}1&(k\notin n\bb{Z}) \\
2^{n}&(k\in n\bb{Z}). \end{cases}$$
\end{cor}
\begin{proof}
By the Peter-Weyl theorem, $L^2(\bb{Z}_n)\cong \bigoplus_{l=0,1,\cdots,n-1} \bb{C}_l$, where $\bb{C}_1$ is the irreducible representation of $\bb{Z}_n$ so that $\gamma$ acts as a fixed primitive $n$-th root $\zeta$ and $\bb{C}_l=\bb{C}_1^{\otimes l}$. Thus $L^2(\bb{Z}_n)\otimes \bb{H}\cong \bigoplus_{l=0,1,\cdots,n-1} \bb{C}_l\otimes \bb{H}$, and hence 
$$\bigwedge^* L^2(\bb{Z}_n)\otimes h
\cong \bigotimes_{l=0,1,\cdots,n-1} \bigwedge^* [\bb{C}_l]\otimes h$$
Note that 
$$\Tr_{(\gamma^k,J)}\bigwedge^* [\bb{C}_l]\otimes h
=\begin{cases}1+\zeta^{2kl} &(k\notin n\bb{Z}) \\
2&(k\in n\bb{Z}). \end{cases}.$$ 
By Lemma \ref{lem3.7}, we obtain the conclusion.

\end{proof}

\begin{thm}\label{thm main thm}
$$\alpha_0(X)+\widetilde{\alpha}_0(X)=
2^{m-2k}\bra{\frac{2^{(\#\Gamma-1)(m-2k+1)}-1}{\#\Gamma}[L^2(\Gamma)]+\rho_{\rm triv}}.$$
\end{thm}
\begin{proof}
By Lemma \ref{lem3.7}, for any $\gamma\in \Gamma$, we obtain
\begin{align*}
&\Tr_{(\gamma,J)}\bra{\alpha(\SW_X)\cdot \prod_{\lambda\neq \rho_{\rm triv}}\bigwedge^*\bra{\lambda\otimes\bra{\ind(D_{\rm Spin}(M))\cdot \dim( V_\lambda^*)h}}
}\\
&\ \ \ =\Tr_{(\gamma,J)}\bra{\prod_{\lambda\neq \rho_{\rm triv}}\bigwedge^*\bra{V_\lambda\otimes\bra{\ind(D_{\rm AHS}(M))\cdot \dim( V_\lambda^*)c}}\cdot \alpha(\SW_M)}.
\end{align*}

By  the Peter-Weyl theorem,
\begin{align*}
&\prod_{\lambda\neq \rho_{\rm triv}}\bigwedge^*\bra{\lambda\otimes\bra{\ind(D_{\rm Spin}(M))\cdot \dim( V_\lambda^*)h}} \\
&=\bbra{\bigwedge^*\bra{\sum_{\lambda\neq \rho_{\rm triv}}\dim( V_\lambda^*)\lambda}\otimes h}^{\ind(D_{\rm Spin}(M))}\\
&=\bbra{\bigwedge^*\bra{[L^2(\Gamma)]-1}\otimes h}^{\ind(D_{\rm Spin}(M))},
\end{align*}
where $1$ is the trivial representation space of $\Gamma$. We have
$$\Tr_{(\gamma,J)}\bra{\bigwedge^*\bra{[L^2(\Gamma)]-1}\otimes h}
=\frac{\Tr_{(\gamma,J)}\bra{\bigwedge^*[L^2(\Gamma)]\otimes h}}
{\Tr_{(\gamma,J)}\bra{\bigwedge^*1\otimes h}}.
$$

Obviously,
$$\Tr_{(\gamma,J)}\bra{\bigwedge^*1\otimes h}
=2.$$

By Corollary \ref{cor lem on cyclic gp},

\begin{align*}
\Tr_{(\gamma,J)}\bra{\bigwedge^*\bbbra{L^2(\Gamma)}\otimes h}
&=\Tr_{(\gamma,J)}\bra{\bigwedge^*\bbbra{L^2(\left<\gamma\right>)}\otimes h}^{\#\Gamma/{\rm ord}(\gamma)} \\
&=\bra{\Tr_{(\gamma,J)}\bigwedge^*\bbbra{L^2(\left<\gamma\right>)}\otimes h}^{\#\Gamma/{\rm ord}(\gamma)} \\
&=\begin{cases}
1 & \gamma\neq e\\
2^{\#\Gamma}& \gamma= e.
\end{cases}
\end{align*}

In the same way,
\begin{align*}
\Tr_{(\gamma,J)}\bra{\bigwedge^*\bbbra{L^2(\Gamma)}\otimes c}
&=\begin{cases}
1 & \gamma\neq e\\
2^{\#\Gamma}& \gamma= e.
\end{cases}
\end{align*}

Thus,
\begin{align}\label{trace}
\Tr_{(\gamma,J)}\alpha(\SW_X)=
\begin{cases}
\Tr_{J}\bra{\alpha(\SW_M)} & \gamma\neq e \\
2^{(\#\Gamma-1)(\ind(D_{\rm AHS}(M)-\ind(D_{\rm Spin}(M))} \Tr_{J}\bra{\alpha(\SW_M)}  & \gamma= e.
\end{cases}
\end{align}


On the other hand, 
thanks to $\Tr_{J}\bra{\alpha(\SW_M)}=2^{m-2k}$,
\begin{align*}
&\Tr_{(\gamma,J)}\bbra{2^{m-2k}\bra{\frac{2^{(\#\Gamma-1)(m-2k+1)}-1}{\#\Gamma}[L^2(\Gamma)]+\rho_{\rm triv}}}\\
&\ \ \ \ \ \ =
\begin{cases}
\Tr_{J}\bra{\alpha(\SW_M)} & \gamma\neq e \\
2^{(\#\Gamma-1)(\ind(D_{\rm AHS}(M)-\ind(D_{\rm Spin}(M))} \Tr_{J}\bra{\alpha(\SW_M)}  & \gamma= e
\end{cases}
\end{align*}
since
$$\Tr_\gamma[L^2(\Gamma)]=\begin{cases}0 & (\gamma\neq e) \\
\#\Gamma& (\gamma= e) \end{cases}.$$

Thanks to character theory of finite groups (see \cite{Ser} for details), we obtain the result.

\end{proof}

\section{$\bb{Z}_p$-coverings}

\subsection{Odd prime cases}

When $\Gamma=\bb{Z}_p$ for an odd prime number $p$, we can completely describe $\alpha(\SW_X)$.
First, we study the common coefficient. Note that $\dim(V_\lambda^*)=1$ for all $\lambda$ because $\bb{Z}_p$ is abelian.

\begin{lem}\label{coeff. does not vanish}
Let $(\gamma,g)\in \bb{Z}_p\times {\rm Pin}(2)$ be one of the following
\begin{itemize}
\item $g\in S^1$ is generic and $\gamma\neq \overline{0}$.
\item $g\notin S^1$.
\end{itemize}
Then, $\tr_{(\gamma,g)}\bra{\prod_{\lambda\neq \rho_{\rm triv}}\bigwedge^*\bra{\lambda\otimes\bra{{2N_\lambda}h+ {M_\lambda}c}}}\neq 0$.
\end{lem}
\begin{proof}

As usual, it suffices to compute $\tr_{(\gamma,g)}(\bigwedge^*\lambda\otimes{h})$ and $\tr_{(\gamma,g)}(\bigwedge^*\lambda\otimes{c})$.

Since $\bigwedge^*\lambda\otimes{h}=1\otimes 1-\lambda\otimes{h}+\lambda^2\otimes 1=(1+\lambda^2)\otimes 1-\lambda\otimes{h}$,
\begin{align*}
\tr_{(\gamma,g)}(\bigwedge^*\lambda\otimes{h})
&= \tr_\gamma(1+\lambda^2)-\tr_\gamma(\lambda)\tr_g(h) \\
&=\begin{cases}1+\lambda(\gamma)^2-\lambda(\gamma)(g+g^{-1}) & (g\in S^1) \\
1+\lambda(\gamma)^2 & (g\notin S^1). 
\end{cases}
\end{align*}
Since $p$ is odd, $\lambda(\gamma)^2\neq -1$. By the assumption, $g$ is generic, and thus $\tr_{(\gamma,g)}(\bigwedge^*\lambda\otimes{h})\neq 0$.

Since $\bigwedge^*\lambda\otimes{c}=1\otimes 1-\lambda\otimes c$,
\begin{align*}
\tr_{(\gamma,g)}(\bigwedge^*\lambda\otimes c)
&=\begin{cases}1-\lambda(\gamma) & (g\in S^1) \\
1+\lambda(\gamma) & (g\notin S^1). 
\end{cases}
\end{align*}
When $g\in S^1$, we notice that $1-\lambda(\gamma)\neq \overline{0}$ since $p$ is prime. Therefore, $1-\lambda(\gamma)\neq 0$. When $g\notin S^1$, $1+\lambda(\gamma)\neq 0$ since $p$ is odd.
\end{proof}

Thus, if $(\gamma,g)$ satisfies the above assumption, the following formula holds.
\begin{align}\label{eqn for SW for finite cover}
\tr_{(\gamma,g)}\bra{\alpha(\SW_X)\cdot \prod_{\lambda\neq \rho_{\rm triv}}\bigwedge^*\bra{\lambda\otimes{{2k}h}}
}=
\tr_{(\gamma,g)}\bra{\prod_{\lambda\neq \rho_{\rm triv}}\bigwedge^*\bra{\lambda\otimes{ ({m+1})c}}\cdot \alpha(\SW_M)}.
\end{align}
We will use this formula in order to study $\alpha(\SW_X)$.

\begin{pro}
$\alpha_k(X)$ vanishes for $k>0$ and $\alpha_0(X)=\widetilde{\alpha}_0(X)$.
\end{pro}
\begin{proof}
Thanks to the formula
$\alpha(\SW_M)=2^{m-2k-1}(1-c)$, we have 
$$\tr_{(\gamma,g)}\bra{\prod_{\lambda\neq \rho_{\rm triv}}\bigwedge^*\bra{\lambda\otimes{ ({m+1})c}}\cdot \alpha(\SW_M)}=0$$
if $g\in S^1$. Therefore, we obtain
$$\tr_{(\gamma,g)}\bra{\alpha(\SW_X)\cdot \prod_{\lambda\neq \rho_{\rm triv}}\bigwedge^*\bra{\lambda\otimes{{2k}h}}
}=0$$
for $\gamma\neq \overline{0}$. Since $\tr_{(\gamma,g)}\bra{\prod_{\lambda\neq \rho_{\rm triv}}\bigwedge^*\bra{\lambda\otimes{{2k}h}}
}\neq 0$, we obtain
$$\tr_\gamma\bra{\alpha_0(X)}-\tr_\gamma\bra{\widetilde{\alpha}_0(X)}+\sum_{k=1}^\infty \tr_\gamma\bra{\alpha_k(X)}(g+g^{-1})^k=0.$$
Since $g$ is generic, $\tr_\gamma\bra{\alpha_k(X)}=0$ for arbitrary $k$ and $\gamma\neq \overline{0}$.

By the result $\alpha(\SW_X)=2^{m_X-2k_X-1}(1-c)$ again, $\tr_e\bra{\alpha_k(X)}=0$ for arbitrary $k$ and $\gamma\neq \overline{0}$.

Combining these results, we notice that $\alpha_0(X)=\widetilde{\alpha}_0(X)$ and $\alpha_k(X)=0$ in $R(\bb{Z}_p)$.
\end{proof}

\begin{thm}\label{main thm 2}
$$\alpha(\SW_X)=\bra{\frac{2^{(m-2k+1)p-2}-2^{m-2k-1}}{p}[L^2(\bb{Z}_p)]+2^{m-2k-1}\rho_{\rm triv}}(1-c).$$
\end{thm}
\begin{proof}
The proof is parallel to the case of  Theorem \ref{thm main thm}.

Applying Lemma \ref{coeff. does not vanish} to Equation (\ref{eqn for SW for finite cover}), we obtain the equlities
$$\tr_\gamma(\alpha_0(X))=
\begin{cases}2^{m-2k-1} & (\gamma\neq \overline{0}),  \\
2^{(m-2k+1)p-2}& (\gamma= \overline{0}).
\end{cases}$$
Thus, $\alpha_0(X)$ is of the form $x\sum_{\rho}\rho+y\rho_{\rm triv}$ for some $x,y\in\bb{Z}$. Since
$$\tr_\gamma\bra{x\sum_{\rho}\rho+y\rho_{\rm triv}}=
\begin{cases}y & (\gamma\neq \overline{0}) \\
px+y & (\gamma= \overline{0}),
\end{cases}$$
$y=2^{m-2k-1}$ and $x=\frac{2^{(m-2k+1)p-2}-2^{m-2k-1}}{p}$. Since $L^2(\bb{Z}_p)=\sum_{\rho}\rho$, we obtain the result.
\end{proof}

\begin{rmk}
$\frac{2^{(m-2k+1)p-2}-2^{m-2k-1}}{p}$ is an integer, because
\begin{align*}
2^{(m-2k+1)p-2}-2^{m-2k-1}
&= 2^{m-2k-1}(2^{(m-2k+1)(p-1)}-1)
\end{align*}
is a multiple of $p$ by  Fermat's little theorem.
\end{rmk}

\subsection{The case of $\bb{Z}_6$}

By combining our result and that of Bryan, we can obtain some information for more general $p$. As a typical example, we study the case when $p=6$.

\begin{dfn}
Let $i_2:\bb{Z}_2\hookrightarrow \bb{Z}_6$ and $i_3:\bb{Z}_3\hookrightarrow \bb{Z}_6$ be the natural embeddings defined by
$$i_2(\overline{1}):=\overline{3},\ \ \ i_3(\overline{1}):=\overline{2}.$$
\end{dfn}

Let $X$ be a  Spin $4$-manifold equipped with a free Spin action of $\bb{Z}_6$. We denote the orbit space $X/\bb{Z}_6$ by $M$. Then we have the equivariant $K$-theory degree $\alpha(\SW_X)\in R(\bb{Z}_6\times {\rm Pin}(2))$. It is of the form 
$$\alpha(\SW_X)=\sum_{i=0}^5 \beta_i\otimes \rho_i,$$
where $\beta_i\in R({\rm Pin}(2))$ and $\rho_i$ is the $i$-th tensor power of the natural representation of $\bb{Z}_6$.
By pulling back it along $i_2$, we obtain an invariant as a $\bb{Z}_2$-manifold
$$i_2^*(\alpha(\SW_X))\in R(\bb{Z}_2\times {\rm Pin}(2)).$$
Similarly, 
$$i_3^*(\alpha(\SW_X))\in R(\bb{Z}_3\times {\rm Pin}(2)).$$

Although we have not determined $\alpha(\SW_X)$, we obtain some information on them from our result and that of Bryan. In order to describe the result, we introduce new symbols.
We have been using symbols $m$ and $k$ for invariants of $M=X/\Gamma$ so far. From now on, we use $m_X=b^+(X)$ and $k_X=\sigma(X)$ to describe all the results, because the ``base manifold'' of $i_2^*(\alpha(\SW_X))$ and that of $i_3^*(\alpha(\SW_X))$ are different. For example, Theorem \ref{bry} is rewritten as 
$$i_2^*(\alpha(\SW_X))=2^{2(m-2k+1)-2-1}[L^2(\bb{Z}_2)](1-c)
=2^{(m_X-2k_X+1)-2-1}[L^2(\bb{Z}_2)](1-c)=2^{m_X-2k_X-2}[L^2(\bb{Z}_2)](1-c).$$
Note that these invariants are related as follows
$$k_X=k\times \#\Gamma=-\frac{\sigma(X/\Gamma)\times \#\Gamma}{16},$$
$$m_X+1=(m+1)\times \#\Gamma=(b^+(M)+1)\times \#\Gamma.$$

Set  
\begin{align*}
& A:= 2^{m_X-2k_X+1-2-1}(1-c)
= 2^{m_X-2k_X-2}(1-c), \\
  & B:=\frac{2^{m_X-2k_X-1}+
 2^{\frac{m_X-2k_X+1}{3}-1}}{3}(1-c), \\
 &C:=\frac{2^{m_X-2k_X-1}-
 2^{\frac{m_X-2k_X+1}{3}-2}}{3}(1-c).
 \end{align*}
 
\begin{cor}\label{Z_6-action}
The equations
\begin{align*}
\beta_2&= A-C-\beta_0+\beta_1,\\
\beta_3&= B-\beta_0, \\
\beta_4&= C-\beta_1, \\
\beta_5&= -A + 2C + \beta_0 - \beta_1
\end{align*}
hold in $R(\Pin(2))$.

 In particular
$\alpha(\SW_X)=\sum_{i=0}^5 \beta_i\otimes \rho_i$
is determined by the two values of elements $\beta_0$ and $\beta_1$.
\end{cor}
\begin{proof}
By Theorem \ref{bry}, $i_2^*(\alpha(\SW_X))=2^{m_X-2k_X+1-3}[L^2(\bb{Z}_2)](1-c)$. It implies 
$$\beta_0+\beta_2+\beta_4=2^{m_X-2k_X-2}(1-c)=A,$$
$$\beta_1+\beta_3+\beta_5=2^{m_X-2k_X-2}(1-c)=A.$$
Similarly, by Theorem \ref{main thm 2},
$$\beta_0+\beta_3=\frac{2^{m_X-2k_X-1}+2^{\frac{m_X-2k_X+1}{3}-1}}{3}(1-c)=B$$
$$\beta_1+\beta_4=\frac{2^{m_X-2k_X-1}-2^{\frac{m_X-2k_X+1}{3}-2}}{3}(1-c)=C$$
$$\beta_2+\beta_5=\frac{2^{m_X-2k_X-1}-2^{\frac{m_X-2k_X+1}{3}-2}}{3}(1-c)=C.$$
Now, it is easy to see the result from these equations.
\end{proof}

\section*{Acknowledgements}
Kato was supported by JSPS KAKENHI 17K18725
and 17H06461.
Takata was supported by JSPS KAKENHI 18J00019 and 21K20320.

Tsuyoshi Kato,
Department of Mathematics,
Faculty of Science,
Kyoto University,
Kyoto 606-8502
Japan. 

E-mail address: {\tt tkato@math.kyoto-u.ac.jp}

\

Doman Takata, 
Faculty of Education Mathematical and Natural Sciences,
Niigata University, 
8050 Ikarashi 2-no-cho, Nishi-ku, Niigata, 950-2181, Japan. 

E-mail address: {\tt d.takata@ed.niigata-u.ac.jp}

\end{document}